\RequirePackage{etoolbox}
\csdef{input@path}{{style/}{graphics/}}
\documentclass[aop,MSNbibl,dvips]{arximspdf}
\makeatletter
   \@ifpackageloaded{graphicx}{}{\usepackage{graphicx}}
\makeatother
\usepackage{mathbh}
\usepackage{url,breakurl}

\doi{10.1214/13-AOP898} 
\volume{43}
\issue{3}
\pubyear{2015}
\firstpage{1399}
\lastpage{1418}
\docsubty{FLA}

\makeatletter
\newcommand{\rrvert}{\vert}
\newcommand{\llvert}{\vert}
\def\mathds{\mathbh}
\def\leqslant{\leq}
\def\geqslant{\geq}
\newcommand{\oo}{\infty}

\renewcommand{\H}{\mathcal{H}}
\newcommand{\M}{\mathcal{M}}
\newcommand{\N}{\mathcal{N}}
\renewcommand{\S}{\mathcal{S}}

\newcommand{\EE}{\mathbb{E}}
\newcommand{\NN}{\mathbb{N}}
\newcommand{\PP}{\mathbb{P}}
\newcommand{\RR}{\mathbb{R}}
\newcommand{\ZZ}{\mathbb{Z}}

\newcommand{\dd}{{\mathrm d}}
\newcommand{\sgn}{\operatorname{sgn}}

\newcommand{\I}{\mathds{1}}

\newtheorem{theorem}{Theorem}
\newproclaim{assumption}[theorem]{Assumption}

\newtheorem{lemma}{Lemma}
\newtheorem{proposition}{Proposition}

\newproclaim{remark}{Remark}

\makeatother

\begin{document}
\begin{frontmatter}

\title{Greedy walk on the real line}
\runtitle{Greedy walk on the real line}

\begin{aug}
\author[A]{\fnms{Sergey}~\snm{Foss}\ead[label=e1]{s.foss@hw.ac.uk}},
\author[B]{\fnms{Leonardo T.}~\snm{Rolla}\ead[label=e2]{leorolla@dm.uba.ar}}
\and
\author[C]{\fnms{Vladas}~\snm{Sidoravicius}\corref{}\ead[label=e3]{vladas@impa.br}}
\runauthor{S. Foss, L. T. Rolla, V. Sidoravicius}
\affiliation{Heriot-Watt University,
Novosibirsk State University and Institute of Mathematics,
Instituto de Matem\'atica Pura e Aplicada and Universidad de Buenos Aires, and
Instituto de Matem\'atica Pura e Aplicada}
\address[A]{S. Foss\\
School of Mathematical \& Computer Sciences\\
Actuarial Mathematics and Statistics\\
Heriot-Watt University\\
Edinburgh\\
EH14 4AS\\
United Kingdom\\
\printead{e1}} 
\address[B]{L. T. Rolla\\
Departamento de Matem\'atica\\
Universidad de Buenos Aires\\
Ciudad Universitaria\\
Capital Federal C1428EGA\\
Argentina\\
\printead{e2}}
\address[C]{V. Sidoravicius\\
Instituto de Matem\'atica Pura e Aplicada\\
Estrada Dona Castorina 110\\
Rio de Janeiro 22460-320\\
Brazil\\
\printead{e3}}
\end{aug}

\received{\smonth{3} \syear{2013}}
\revised{\smonth{10} \syear{2013}}

%
\begin{abstract}
We consider a self-interacting process described in terms of a single-server
system with service stations at
each point of the real line.
The customer arrivals are given by a Poisson point
processes on the space--time half plane.
The server adopts a greedy routing mechanism, traveling toward the
nearest customer, and ignoring new arrivals while in transit.
We study the trajectories of the server and show that its asymptotic position
diverges logarithmically in time.
\end{abstract}

%
\begin{keyword}[class=AMS]
\kwd{60K25}
\kwd{60K35}
\kwd{90B22}
\end{keyword}
\begin{keyword}
\kwd{Greedy policy}
\kwd{self-interaction}
\kwd{long-term behavior}
\kwd{stability}
\end{keyword}
\end{frontmatter}

\section{Introduction}
\label{sec:sec1introduction}

We consider a self-interacting process described in terms of a single-server
system with service stations at
each point of the real line.
The system is described as follows.
Initially, there is a Poisson field of customers in $\RR$ with unit
intensity and the server starts
at $x=0$.
Customers arrive as a Poisson point process in the space--time
$\RR\times\RR_+$ with intensity $\lambda>0$.
When not serving, the server chooses the nearest customer and travels toward
it at speed $0 < v \leqslant\oo$, ignoring new arrivals.
The service then takes $T$ units of time with $\EE T=1$,
after which the customer leaves the system.
This is a common example of a routing mechanism that depends on the system
state, and targeting the nearest customer is known as a \emph{greedy strategy}.

The particular interest in customer-server systems in continuous space
stems from their transparent description of large systems with spacial
structure,
in contrast with finite systems where phenomenological properties are
often obscured by combinatorial aspects of the model.
However, systems with greedy routing strategies in the continuum are extremely
\emph{sensitive to microscopic perturbations}, and their rigorous study
represents a challenging problem;
a topic that has been active
for almost three
decades \cite{altman-levy-94,bertsimas-vanryzin-91,bordenave-foss-last-11,coffman-gilbert-87,foss-last-96,kroese-schmidt-94,kroese-schmidt-96,leskela-unger-11,litvak-adan-01,meester-quant-99,robert-10,rojasnandayapa-foss-kroese-11}.

The system described above arises naturally in the
question of stability of a greedy server on the circle $\RR/\ZZ$.
It was conjectured in \cite{coffman-gilbert-87} that the greedy server
on $\RR/\ZZ$ is stable when $\lambda<1$, regardless of the speed $v$.
This~was verified only under light-traffic assumptions, that is, for large
enough $v$
given $\lambda$~\cite{kroese-schmidt-96},
and for the greedy server
on a discrete ring
$\ZZ/n\ZZ$ \cite{foss-last-96,foss-last-98,meester-quant-99,schasberger-95}.\hskip.2pt\footnote{Stability
was also shown for
a number of other finite graphs and a broader class of service
strategies \cite{schasberger-95}, as well as
several nongreedy policies \cite{kroese-schmidt-94},
a gated-greedy variant on convex spaces \cite{altman-levy-94}
and random nongreedy servers on general spaces \cite{altman-foss-97}.
See
\cite{rojasnandayapa-foss-kroese-11} for a recent review.}
Yet, discrete models have not been able to grasp the microscopic
nature of the greedy mechanism in continuous space, and there are major
obstacles in extrapolating any approach based on a discrete approximation.

On the other hand, stability under the same conditions
is known to hold for the \emph{polling
server} on $\RR/\ZZ$, that is, the server whose strategy is to always travel
in the same direction \cite{kroese-schmidt-92}.\hskip.2pt\footnote{Several other state-independent strategies have been analyzed and, in
particular, the following two:
after each service, the server decides to move next in a direction
chosen at random;
the server follows the path of a Browning motion.
Stability under the same conditions holds in both cases, see~\cite{kroese-schmidt-92}.}
Simulations indicate not only that the greedy server is
stable, but also that under heavy traffic conditions its
dynamics resembles that of the
polling server \cite{coffman-gilbert-87}.

This prompts a detailed study of its \emph{local behavior},
and it is natural to describe it with a model on an \emph{infinite line}.
The model on the line is clearly different from that on the circle as a
queueing system, for its total arrival rate is infinite.
The study of the former is rather intended to give \emph{mathematical
insight} on the server's local motion, as is confirmed in
\cite{rolla-sidoravicius-b}.
It was shown in \cite{kurkova-menshikov-97} that the position of the greedy
server on $\ZZ$ is \emph{transient}.
But again, the behavior in this case is governed by averaging effects
inside each discrete cell overcrowded by waiting customers, and its
understanding is of little help for the continuous-space system.
The goal of this paper is to study the greedy server on $\RR$.

The main difficulty in studying this model is due to the
interplay between the server's motion and
the environment of waiting customers that surround it.
This interplay is given by the
\emph{interaction} at the \emph{microscopic level} resulting from the
greedy choice of the next customer and the removal of those who have been
served.
The server's path is locally \emph{self-repelling}, since the removal
of already
served
customers makes it less likely for the greedy server to take the next
step back into the recently visited regions.

There are several deep studies of processes which in different ways are
self-repelling.
This includes examples of self-interacting walks such as
the random walk avoiding its past convex
hull \cite{angel-benjamini-virag-03,zerner-05},
the prudent walk \cite{beffara-friedli-velenik-10,bousquet-melou-10},
the ``true'' self-avoiding walk \cite{toth-95,toth-99} and
excited random walks \cite{benjamini-wilson-03}.
In the continuum setup, one has
the self-interacting diffusion with repulsion \cite{mountford-tarres-08},
the perturbed Brownian
motions \cite{carmona-petit-yor-98,chaumont-doney-99,davis-96,davis-99,perman-werner-97},
the excited Brownian motions \cite{raimond-schapira-11}
and
random paths with bounded local time \cite{benjamini-berestycki-10}.\hskip.2pt\footnote{See the introductions
of \cite{mountford-tarres-08,raimond-schapira-11} for
concise reviews on these models, \cite{merkl-rolles-06} for a review on
reinforced walks, and \cite{pemantle-07} for a comprehensive survey on the
field up to 2007.
}
It was clear since these models were introduced that they
could not be treated via standard methods and tools.
Despite the existence of a few disconnected techniques
that have proved useful in specific situations,
this rich research field still lacks a systematic basis of
study.\footnote{Except for the family of universality classes
given by the Schramm--L\"{o}wner Evolutions \cite{schramm-00},
which include $2$-dimensional loop-erased random
walk \cite{lawler-schramm-werner-04-1} and several other
models \cite{lawler-schramm-werner-04,smirnov-01,smirnov-10}.}
A lot remains to be understood even in dimension $d=1$,
and in particular
none of the known techniques seems to be applicable in our case.

The future evolution of the greedy server's
position is of course influenced by its previous path.
But unlike the above models, here there is no direct prescription of
such influence in terms of occupation times.
A similar situation occurs while defining
the ``true self-repelling motion'' in
$d=1$ \cite{toth-werner-98}, although the authors show
that its
evolution depends on the occupation times, and moreover that such
dependency is
local.

Notice that, for the greedy server, ``self-repulsion''
does not imply immediately ``repulsion toward $\oo$,'' since the
server is
allowed to backtrack, in which case it starts being repelled back
toward the origin.\footnote{A difficulty similar in spirit was faced in
\cite{mountford-tarres-08},
where it was proved that a certain diffusion with self-repelling potential
has a power-law asymptotic behavior.}
Another particular feature of this model is an inverse relation between
the strength of self-repulsion (measured by the bias in the probability
that the
server takes the next step backward) and the average speed of the server.
The attraction felt by the server upon reaching unexplored
regions is increasing in time, due to the accumulation of customers
that keep arriving throughout the whole evolution, but at the same time
this high concentration of customers causes the subsequent traveled distances
to become shorter at the same proportion.

In this paper, we introduce a framework based on a randomized
representation of
the customers environment as
viewed from the
server
(namely, it learns only the information that is necessary and
sufficient to
determine the next movement, and the positions of further waiting customers
remain unknown).
This allows a fine description of the system behavior.
As a consequence of this approach, we show transience and describe the
server's asymptotics,
setting up an old question in the field
(stated, e.g., as Open Problem 4 in \cite{rojasnandayapa-foss-kroese-11}).

\begin{theorem}
\label{thm:transientlog}
Let ${\mathcal S}_t$ denote the server position at time $t$.
Assume that $\EE e^{\alpha T}<\oo$ for some $\alpha>0$.
Then for any $v>0$ and $\lambda>0$ the greedy server on the real line
is transient.
Moreover,
\[
\frac{\mathcal S_t}{\lambda^{-1}\log t} \to\pm1
\]
with probability $1/2$ each.
\end{theorem}

\begin{remark}
In our approach, it is important that the arrivals form a Poisson process
in space--time, and that they are independent of the service times.
\end{remark}

\begin{remark}
Assume that at time $0-$ the set of waiting customers is
distributed as a Poisson point process with intensity $\mu(x)\,\dd x$,
for some
nonnegative bounded measurable function $\mu$ with $\int\mu=\infty$,
and with an additional deterministic finite set of points.
Then Theorem~\ref{thm:transientlog} remains true (with
essentially the same proof), except for the lack of symmetry in the
probabilities of ${\mathcal S}_t$ diverging to $+\infty$ or $-\infty$.
\end{remark}

\begin{remark}
There is a dynamic version of the greedy server, where new arrivals
are not ignored while the server is traveling.
This variation might be studied by similar arguments, but the dynamic
mechanism introduces some extra complications that will not be
considered here.
\end{remark}

\begin{remark}
The assumption of constant speed is natural in several contexts where
the terminal speed is quickly achieved, but not crucial in our construction.
In fact, for a server moving with constant acceleration, or any other
mechanical constraints (which restarts after each service), mild
modifications of our method yield the same results.
Notice that the value of $v$ plays no role in Theorem~\ref{thm:transientlog}.
\end{remark}

Heuristically, the asymptotics described by Theorem~\ref
{thm:transientlog} is
what one should expect to happen, assuming that the server will
indeed move most of
the times in the same direction.
Suppose that
all of the first $N$ customers were found to the right of the server.
The typical
distance between the server and the next customer to the right is
about $\frac{1}{N}$, because customers have been arriving to this region
for about $N$ time units.
To the
left of the server, there are regions of size about $\frac{1}{N-1}$,
$\frac{1}{N-2}$, $\frac{1}{N-3}$, etc., where the arrival of customers
is rather
recent: they must have happened during the last $1$, $2$, $3$, etc.,
units of
time.
If the server is eventually moving only to the right (or the
excursions to the left
are very sparse in time), the server position $S_N$ should therefore
diverge as
$\log N$.\footnote{This heuristics is confirmed by the asymptotic behavior of
a continuous model in $\RR$ where there is no greedy mechanism
and the server is always moving to the right \cite{kurkova-96},
as well as for the greedy server on $\ZZ$ \cite{kurkova-menshikov-97}.}

However, the probability that the next customer is
found to the left of the server is about
$\frac{C}{N}$, which implies that it will happen some time in the future.
In fact, the server will make an excursion
of length $\frac{c}{N}$ to the left for infinitely many
$N$, for any constant $c$, in contrast with its discrete variant.
Nevertheless, the probability that the two next customers are both
to the left is
about $\frac{C}{N^2}$.
One may thus push this argument and show that indeed, with positive probability,
the system will never produce microscopic scenarios capable of
causing important changes in the server's course.

To make the above observation rigorous, we introduce a dynamic block
construction, where the block sizes are increasing at each step, and
combine it with a renewal argument.
The size $\ell_k$ of the blocks (groups of sequentially served
customers) should
increase slow enough so that the cleared region left by a block is wide enough
to support the next one, but fast enough so that the probability of atypical
gaps inside the blocks is summable in $k$.
It turns out that a growth $\ell_k \sim k^\eta$, with
$0<\eta<\frac{1}{2}$, works well for this purpose.

This paper is divided as follows.
In Section~\ref{sec:potential}, we present the evolution of the
customers environment as
viewed from the server, and study its properties.
We state Proposition~\ref{prop:blocks} about the behavior of the
greedy server
on the
real line at specific times (a block argument) and show how it implies
Theorem~\ref{thm:transientlog}, and in Section~\ref{sec:blocks} we prove
Proposition~\ref{prop:blocks}.
In these sections, we consider the case where $T$ is deterministic and
$v=\infty$.
This case contains the most important features of the construction
and the block argument, but is simpler to present.
The general case is considered in Section~\ref{sec:finitespeedrandomservice}.

\section{The process viewed from the server}
\label{sec:potential}

We consider a particular construction of the initial state by assuming
that there are arrivals during $t \in[-1,0]$, before the service starts
at $t=0$.
This is of course equivalent to simply starting at $t=0$ with a Poisson
field of
points.
Let $\nu$ denote the random set of arrivals in $\{(x,t)\dvtx x\in\RR, t >
-1\}$.

We want to construct the process by following a progressive exploration
of the
space--time until finding the mark $(x^*,t^*)\in\nu$ corresponding to the
nearest waiting customer, getting as
little information as possible about $\nu$.
The server is thus unaware of existing customers further than the
nearest one,
and keeps record of the last time when each point in space was
explored in the seek of waiting customers.

For the reasons mentioned above, here we consider the case $T=1$ and
$v=\oo$,
and postpone the general case to Section~\ref{sec:finitespeedrandomservice}.
Thus, the server's position $\S_t$ remains constant on intervals $t\in
[n-1,n)$.
By rescaling space, we can assume $\lambda=1$.

Starting at $t=0$, each region on the space has potentially witnessed the
arrival of customers during $1$ unit of time.
The first customer is then found at an exponentially-distributed
distance, to
the left or to the right with equal probabilities.
Discovering its position reveals the presence of a point in $\nu$, as
well as a
region where $\nu$ has no points.
For the second customer, there is a region in space that has potentially
witnessed the arrival of customers during $1$ unit of time (namely, the
region explored on the previous step), and the
complementary region has not been queried during the last $2$ units of time.
The position of the third customer is already more involved, and the positions
of both of the previous customers are important in determining the
regions where
$\nu$ is still unexplored.
Yet there is a general description which is amenable to study, which motivates
the construction described hereafter and depicted in Figure~\ref{fig:potentialchain}.

A \emph{potential} is a piecewise continuous function $u\dvtx \RR\to\RR$
such that
there is a unique point $x_*=S(u)$ where it attains its maximum
$\M=\M(u)=u(x_*)$.

Given a pair of positive numbers $w=(E,U)$, where $0<E<\infty$ and \mbox{$0<U<1$},
we define the operator $\H_w$ as follows.
Let $u$ be given and take $z>0$ as the unique number such that
\[
\int_{x_*-z}^{x_*+z}(\M-u)\,\dd x = E.
\]
Let
\[
a=\M-u(x_*-z),\qquad b=\M-u(x_*+z),
\]
choose
\[
x^{*} = %
\cases{ x_* - z, & \quad $\mbox{if }\displaystyle U \in\biggl(0,
\frac{a}{a+b}\biggr]$,\vspace*{2pt}
\cr
x_* + z, &\quad  $\mbox{if }\displaystyle U \in\biggl(
\frac{a}{a+b},1\biggr)$, } %
\]
and finally
%
\begin{equation}
\label{eq:randomupdate} \bigl(\H_w (u) \bigr) (x) = %
\cases{ \M+1, &\quad
$x=x^*$,\vspace*{2pt}
\cr
\M, &\quad $x\in[x_*-z,x_*+z],x\ne x^*$,\vspace*{2pt}
\cr
u(x), &\quad $\mbox{otherwise.}$} %
\end{equation}
Notice that $\M(\H_w (u))=\M(u)+1$, $S(\H_w (u))=x^{*}$, and
$\int_\RR[\H_w (u)-u]\,\dd x = E$.
Moreover, if $u$ is unimodal, then $\H_w (u)$ is also unimodal.

Start with $u_0^0\dvtx \RR\to\RR$ given by $u_0^0(x)=\delta_{0}(x)-1$, let
$(E_n)_n$
and $(U_n)_n$ be independent i.i.d. sequences of exponential and uniform
random variables, and write $w_n=(E_n,U_n)$.
Let $u_n^0 = \H_{w_n} (u_{n-1}^0)$ and write $S_n=S(u_n^0)$.

\begin{lemma}
\label{lemma:coupling}
The sequence $(S_n)_{n=1,2,\ldots}$ defined above has the same
distribution as
the sequence $({\mathcal S}_{t-})_{t=1,2,\ldots}$ given by the positions
of the
greedy server at integer times.
\end{lemma}

The lemma follows from the properties of the Poisson point
process $\nu$ on
\[
\Gamma_u:= \bigl\{(x,t)\dvtx x\in\RR, u(x) \leqslant t < \infty\bigr
\} \subseteq \RR^2.
\]
Indeed, consider a progressive exploration at the left and right
vertical boundaries of the
continuously-expanding region $\{(x,t)\dvtx x_*-z \leqslant x \leqslant
x_*+z, u(x)
\leqslant t \leqslant\M(u) \}$ as $z$ increases, starting from $0$ until
finding the first point
$(x^{*},t^{*})$ of~$\nu$.
The variable $E$ is given by the area of the explored region.
The variable $U$ is related to the position of $(x^{*},t^{*})$ on the
union of
the two disjoint vertical intervals where this region is growing, and
is given
by $t^* = \M(u) - |(a+b)U - a|$, $x^*=x_*+z \cdot\sgn[(a+b)U - a]$.
By the properties of a Poisson point process, $E$ and $U$ are
independent of
each other, distributed as standard exponential and uniform variables,
regardless of how $\nu$ had been explored outside $\Gamma_u$.

We now consider some properties of the operators $\H$.
Let $\theta_z u = u(z+\cdot)$.
For any potential $u$, any number $c$, and any point $z$, $\M(\theta_z
u+c)=\M(u)+c$ and $S(\theta_z u+c)=S(u)-z$.
It follows from the definition of $\H$ that
%
\begin{equation}
\label{eq:H-only-sees-shape} \H_w(\theta_z u + c) =
\theta_z \H_w(u) + c.
\end{equation}

A potential $u$ is said to be \emph{centered} if $S(u)=0$ and $\M(u)=0$.
Define the operator $\Theta^u(\cdot)=\theta_{S(u)}(\cdot)-\M(u)$,
so that $\Theta^u(u)$ is centered.
For given potentials $u$ and $\tilde{u}$,
%
\begin{equation}
\label{eq:Theta-composed} \Theta^{\Theta^{u}(\tilde{u})} \circ\Theta^{u} =
\Theta^{\tilde{u}}.
\end{equation}

The natural \emph{shifts} in this evolving sequence of potentials
$(u^0_n)_{n\geqslant0}$ is given for each $k$ by
$(u^{k}_n)_{n\geqslant0}$
defined as $u^{k}_n:= \Theta^{u^{k-1}_1} (u^{k-1}_{n+1})$.
Expanding this recursion and using (\ref{eq:Theta-composed}) yields
%
\begin{equation}
\label{eq:ukn} u^{k}_n = \Theta^{u^{k-1}_1} \bigl(
\Theta^{u^{k-2}_1} \bigl(u^{k-2}_{n+2}\bigr)\bigr) =
\Theta^{u^{k-2}_2} \bigl(u^{k-2}_{n+2}\bigr) = \cdots=
\Theta^{u^{0}_k} \bigl(u^{0}_{n+k}\bigr).
\end{equation}
In particular, $u^{k}_0 = \Theta^{u^{0}_k}(u^{0}_{k})$.
Writing
\[
\H^k_n = \H_{w_{k+n}} \circ\cdots\circ
\H_{w_{k+2}} \circ\H_{w_{k+1}},
\]
it follows from (\ref{eq:H-only-sees-shape}) that $u^{k}_n = \H^k_n(u^k_0)$.
Therefore, $u^k_0$ is determined by $w_1,w_2,\break g\ldots, w_k$, whereas
$(u^{k}_n)_{n\geqslant0}$ is determined by $w_{k+1},w_{k+2},\ldots$
and $u^k_0$
itself.

The above properties imply that the evolution of $(u^0_n)_n$ is a homogeneous,
translation-invariant, height-invariant, Markov chain in the space of
potentials.
At any moment $k$, we
can take $u^0_k$ and move the axes so that the origin is placed on its maximum
(i.e., apply $\Theta^{u^{0}_k}$), obtaining $u^k_0$, and from this
point on the evolution of $(u^k_n)_n$ is independent of $(u^0_1,\ldots,u^0_k)$,
and obeys the same transition rules.
Moreover, $(u^k_n)_n$ is related to $(u^0_{k+n})_n$ by $u^k_n =
\Theta^{u^{0}_k} (u^0_{k+n})$.
An example depicting this construction is shown in
Figure~\ref{fig:potentialchain}.

\begin{figure}

\includegraphics{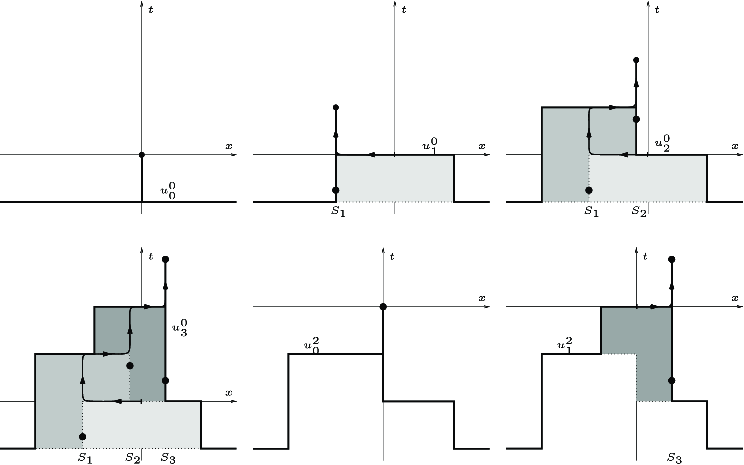}

\caption{Revealing three points of $\nu\subseteq\RR\times
(-1,\infty
)$ to
determine the greedy server's first steps.
Before starting, the configuration is unknown on the whole $\nu
\subseteq\RR
\times(-1,\infty)$, represented by the graph $u^0_0(x)=\delta_0(x)-1$.
The nearest customer found at time $0$ corresponds to the bold point $(x^*,t^*)$
in the
second plot (middle above), where the graph of $u^0_1$ covers the
region that
had to be explored in order to find $(x^*,t^*)$.
After serving this customer, the point in $\nu$ corresponding to the nearest
customer corresponds to a new bold point appearing in the third plot,
where the
graph of $u^0_2$ covers the total region explored in these two steps.
The server's trajectory is depicted by the arrowed, curly path, and
consists of
unit service times alternated with instantaneous space displacements.
The fourth plot (below, left) shows the three points of $\nu$
determining the
construction of $u^3_0$, the region of $\Gamma_{u^0_0}$ explored and
the path
performed by the server during the interval $[0,3)$.
The fifth and sixth plots (below, center and right) depict the
Markovian nature
of this procedure.
At the second customer's departure time, we place the axes on the
maximum of
$u^0_2$, obtaining $u^2_0$.
Notice that in this picture there is no record of the past trajectory
and the
location of the other two points also called $(x^*,t^*)$.
It turns out that the potential is enough in order to determine the future
evolution, and we find the same point $(x^*,t^*)$ corresponding to the next
customer.}
\label{fig:potentialchain}
\end{figure}

%
%
%
%
%

This motivates us to define the evolution of the greedy server model starting
from any centered potential $u$ as the initial $u_0^0$, not necessarily given
by $\delta_0-1$.
Namely, the system starts at $t=0$, with customer arrivals in
space--time given
by a Poisson point process $\nu$ on $\Gamma_u$.
We denote its law by $\PP^u$.

In the proof of Theorem~\ref{thm:transientlog}, we only use two
properties of
$u^0_0(x)=\delta_0(x)-1$.
We say that a potential $u$ is \emph{unimodal} if $u$ is nondecreasing on
$(-\infty,S(u))$ and nonincreasing on $(S(u),+\infty)$.
We say that a potential $u$ is \emph{bounded} if $m(u):= \M(u) -
\inf_{x \in
\RR}u(x)$ is finite.
Each of these conditions is preserved by the operators of the form $\H
_\omega$.
Since they are also preserved by $\theta_z$ and $u\mapsto u+c$,
starting from
$u^0_0(x)=\delta_0(x)-1$ the potentials $u^k_n$ are unimodal and
bounded for any
$k$ and $n$.
Proposition~\ref{prop:blocks} below, though, uses only unimodality.

\textit{Grouping customers.}
Theorem~\ref{thm:transientlog} is proved by grouping customers as in the
following proposition.
Write
\[
\ell_j=\bigl\lceil12j^{1/4} +1 \bigr\rceil, \qquad j=1,2,3,\ldots.
\]

%
\begin{proposition}
\label{prop:blocks}
Let $(S_n)$ be a greedy random walk generated by a centered unimodal
initial potential $u$.
For any $\varepsilon>0$, there exists
$\delta>0$ depending on
$\varepsilon$ but not on the potential $u$, and
a sequence of stopping times
$L_0,L_1,L_2,L_3,\ldots$ with the following properties.

Let $\sigma= \sgn S_1$, $Z_j=\sigma S_{L_j}$, $N_j = L_j-u(S_{L_j})$,
$Q_j=L_{j+1}-L_j$ and $X_{j} = Z_{j+1} - Z_{j}$.
Then, with probability at least $\delta$, we have, for all
$j=1,2,3,\ldots,$
%
\begin{eqnarray}
\label{eq:groupbehavior} %
\cases{ Q_j = \ell_j \mbox{ or
} \ell_j+1, &\vspace*{2pt}
\cr
X_j^- \leqslant
X_j \leqslant X_j^+, &\vspace*{2pt}
\cr
Z_{j-1} <
\sigma S_n < Z_{j+1}, &\quad $\mbox{for } L_j
\leqslant n < L_{j+1},$} %
\end{eqnarray}
where
\[
X_j^- = (1-\varepsilon) \frac{\ell_j-1}{N_{j+1}}\quad \mbox{and}\quad
X_j^+ = (1+\varepsilon) \frac{\ell_j}{N_j}.
\]
\end{proposition}

In words,
$Z_{j}$ is the server
position after serving $L_j$ customers (in case $\sigma=+1$, otherwise
the picture is mirrored), $N_{j}$ is
the discontinuity of the potential at the server's position at this moment,
and finally
$X_j$ measures the displacement in space after serving the next $Q_j$
customers.

The above proposition is proved in the next section.
Let us show how it implies the main result.

\begin{pf*}{Proof of Theorem~\ref{thm:transientlog} for $v=\oo$ and $T=1$}
Let $\varepsilon$ be any positive number.
The system starts at time $n_0=0$ from the potential $u^0_0$, and by
Proposition~\ref{prop:blocks}, with probability at least $\delta$
the events (\ref{eq:groupbehavior})
hold for all $j$, for some sequence of stopping times $L_j$.
If it does not hold for all $j$, let $j_*$ be the first $j$ for which
condition~(\ref{eq:groupbehavior}) is violated, and call $n_1=L_{j_*+1}$.
Whether (\ref{eq:groupbehavior}) occurs or not is determined by
$(u^0_{n})_{n=0,1,\ldots,L_{j+1}}$.
Since $L_{j+1}$ is a stopping time,
defining $n_1=\oo$ on the event that (\ref{eq:groupbehavior}) is
satisfied for all $j$, we have that $n_1$ is also a stopping time.
Therefore,
at time $n_1$ the system restarts from some
unimodal bounded potential $u^{n_1}_0$, ignoring the past history, that is,
conditioned on $n_1$ and $u^0_{n_1}$, $(u^{n_1}_n)_{n\geqslant0}$ is
distributed
as $\PP^{u^{n_1}_0}$.
Again, starting from such potential there is probability at least
$\delta$
that (\ref{eq:groupbehavior}) holds for all $j$,
with $(u^0_n)_n$ replaced by $(u^{n_1}_n)_n$.
It thus takes at most a geometric number of restarts (with
parameter $\delta$) to get a success, so there is an a.s. finite time
$n_*$ such that condition (\ref{eq:groupbehavior}) holds for all $j$,
with $(u^0_n)_n$ replaced by $(u^{n_*}_n)_n$.
Notice that $\sigma$ takes a possibly new value at each attempt.

We write $a \sim_\varepsilon b$ if $\limsup\llvert \frac{a}{b}-1\rrvert
\leqslant
\varepsilon$ and $a \sim b$ if $\frac{a}{b}\to1$.
By definition of $\ell$ and $L$, we have $L_j \sim\frac
{48j^{5/4}}{5}$ and
$\ell_j / L_j \sim\frac{5}{4j}$.
Now, by construction of $N$, $L_j \leqslant N_j \leqslant L_j + m(u_0^{n_*})$
and, therefore, $N_{j+1} \sim N_j \sim L_j$.
Finally, assuming that (\ref{eq:groupbehavior}) holds for
all $j$, $X_j \sim_\varepsilon\ell_j/L_j$.

But $Z_{j+1} = Z_0 + \sum_{i=1}^j X_i$, and putting these all together gives
\[
Z_{j+1} \sim_\varepsilon\tfrac{5}{4}\log j.
\]

Finally, the position $S_n$ is given by $S_n = S_{n_*} +\sigma Z_j$ at
times $n$
satisfying $n = n_* + L_j$ and, therefore,
\[
S_{n} \sim_\varepsilon\sigma\log n \qquad\mbox{a.s.}
\]
Since $\varepsilon$ was arbitrary,
\[
\frac{S_{n}}{\sigma\log n} \to1 \qquad\mbox{a.s.}
\]
and using Lemma~\ref{lemma:coupling} this completes
the proof of Theorem~\ref{thm:transientlog} for $v=\oo$ and $T=1$.
\end{pf*}

\section{Block argument}
\label{sec:blocks}

In this section, we prove Proposition~\ref{prop:blocks}.
Let $0<\varepsilon<\frac{1}{2}$.
Here and in the next section, each time $C$ or $c$ (resp., $c_\varepsilon
$ or $C_\varepsilon$) appears, it
denotes a different constant (resp., function of $\varepsilon$) that is
positive,
finite and universal.
We write $a\vee b$ for $\max\{a,b\}$.

We are going to define the event $A_j$ that step $j$ is successful.
For each $j$, the occurrence of $A_j$
implies (\ref{eq:groupbehavior}),
and we will show that there exists a sequence $p_0,p_1,p_2,\ldots,$ depending
only on $\varepsilon$, such that
%
\begin{equation}
\label{eq:successprob} \PP^u ( A_j | A_{j-1},A_{j-2},
\ldots,A_0 ) \geqslant p_j
\end{equation}
and
%
\begin{equation}
\label{eq:prodpositive} \prod_{j=0}^{\oo}
p_j > 0.
\end{equation}
For the latter, we show that $p_j$ increases fast enough so that
$1-p_j$ is
summable, and that $p_j>0$ for all $j$.
Let us drop the superscript $0$ in the potentials $u^0_n$.

We start with $j=0$, omitted in the statement of Proposition~\ref{prop:blocks}.
Define $\ell_0=1$, and take $L_0=0$, $Z_0=S_0=0$, and $L_1=Q_0=1$.
We choose $\sigma=\sgn(S_1)$.
Let $Z_1=X_0=\sigma S_1=|S_1|$, and $N_1 = L_1-u(\sigma Z_1)$.
We say that step $0$ is \emph{successful} if
%
\begin{equation}
\label{eq:triggering} X_0 \geqslant X_0^-:= \frac{4}{N_1},
\end{equation}
otherwise we declare step $0$ to have \emph{failed} and stop.
The next steps $j=1,2,3,\ldots$ are described assuming for simplicity
that $\sigma=+1$.

Suppose that steps $0,1,2,\ldots,j-1$ have been successful and start from
$u_{L_j}$.
Step $j$ may be successful in two situations.
First, if each of the next $\ell_j$ customers
$S_{L_j+1},S_{L_j+2},\ldots,S_{L_j+\ell_j}$ satisfy $S_n > S_{n-1}$,
in which
case we take $Q_j=\ell_j$.
Second, if there is one $\tilde n\in\{L_j+1,\ldots,L_j+\ell_j\}$ such that
$S_{\tilde{n}} < S_{\tilde{n}-1}$, and $S_n > S_{n-1}$ for all
$n\in\{L_j+1,\ldots,L_j+\ell_j,L_j+\ell_j+1\}$ except $\tilde{n}$,
in which
case we take $Q_j=\ell_j+1$.
If none of these two happen, we declare step $j$ to have \emph{failed}
and stop.
Otherwise, in either of the above two cases we say that step $j$ is
\emph{successful} if~(\ref{eq:groupbehavior}) is satisfied.\footnote{We could have taken $Q_j$ always equal $\ell_j+1$ and have a simpler
proposition with nonrandom times $L_j$. In this case, we would define
step $j$ to be successful if $S_{n} > S_{n-1}$ for all $j=L_j+1,\ldots,L_j+\ell_j+1$ except for possibly one $\tilde{n}$ in $L_j+1,\ldots,L_j+\ell_j$. This would result in a simpler statement but less robust
proof. More precisely, the simple estimate (\ref{eq:twopoints}) below
would not suffice, and a special treatment would be needed for the last
point $S_{L_j+\ell_j+1}$.}

Notice that, for $j\geqslant1$, if step $j-1$ is successful we have
%
\begin{equation}
\label{eq:staircase} %
\cases{ \M(u_{L_{j}}) = L_{j},&
\vspace*{2pt}
\cr
u_{L_{j}}(x) = u_0(x) \leqslant
L_{j} - N_{j}, & \quad $\mbox{for } x > Z_{j}$,
\vspace*{2pt}
\cr
u_{L_{j}}(x) \geqslant L_{j} - Q_{j-1},
& \quad $\mbox{for } Z_{j} - X_{j-1}^- < x < Z_{j}$.}
\end{equation}

Having described the grouping steps, it remains to show (\ref{eq:successprob})
and (\ref{eq:prodpositive}).

Recall from the previous section that, once $u_n$ is fixed, the
position of the
next customer $S_{n+1}$ is determined by a pair $E_{n+1},U_{n+1}$ of
exponentially- and uniformly-distributed random variables, or
alternatively by
the Poisson point process $\nu$ restricted to the region $\{(x,t)\dvtx u_{n}(x)
< t \leqslant\M(u_{n}) \}$.

We start with $j=0$.
In this step, we pay a \emph{finite price} $p_0$ to produce a potential
which exhibits a \textsl{plateau} with convenient shape, namely a potential
satisfying (\ref{eq:triggering}).
Recall that $E_1$ and $U_1$ are the exponential and
uniform random variables used in order to produce $u_1$ from $u_0$.
Consider the event that $E_1$ and $U_1$ satisfy the following two requirements.
The first requirement is that $E_1>8$.
The second one is that, given $E_1$, the variable $U_1$ lies on the
largest interval among
$[0,\frac{a}{a+b}]$ and $[\frac{a}{a+b},1]$; see (\ref{eq:randomupdate}).
This is when $\sigma$ is determined.
In the worst case, this interval has length $\frac{1}{2}$, whence the
probability that both conditions are satisfied is at least
$p_0 = \frac{1}{2}e^{-8} > 0$.
The requirement for $U_1$ implies that $u(S_1) \leqslant u(-S_1)$.
Hence, by monotonicity of $u_0$, the occurrence of
the above event implies that
\begin{eqnarray*}
8 &<& \int_{-X_0}^{+X_0} -u(x)\,\dd x \leqslant\int
_{-X_0}^{+X_0} \max_{[{-X_0},{+X_0}]} (- u) \,\dd x
= - 2 X_0 u(\sigma X_0) = - 2 X_0
u(S_1) \\
&\leqslant& 2 X_0 N_1.
\end{eqnarray*}
The above inequality implies $A_0$ and, therefore,
$\PP^u(A_0) \geqslant p_0 > 0$.

Fix some $j=1,2,3,\ldots.$
We will describe events $B_1,B_2,B_3$, omitting
the dependency on $j$, such that $B_1\cap B_2\cap B_3$ implies $A_j$.
The conditional probability of $B_1\cap B_2\cap B_3$ given $u_{L_j}$
can be
bounded from below by some number $p_j$ that does not depend on the potential
$u_{L_j}$ as long as it satisfies (\ref{eq:staircase}).
This in turn implies (\ref{eq:successprob}).

We stress that, even though the knowledge about these events inconveniently
provides more information about $\nu$ than needed in determining $u_{L_{j+1}}$,
we only study them with the purpose of estimating the probability of
$A_j$.
The occurrence of the latter is entirely determined by
$u_{L_j},u_{L_j+1},u_{L_j+2},\break \ldots,u_{L_{j+1}}$.

We consider the evolution given by the point process $\nu$ itself
rather than
the construction specified in (\ref{eq:randomupdate}).
We write $\nu_i = \nu\cap R_i$, where
\begin{eqnarray*}
R_1 &=& \bigl\{ (x,t) \dvtx x > Z_j, u(x)<t\leqslant
L_j \bigr\},
\\
R_2& =& \bigl\{ (x,t) \dvtx Z_j < x <
Z_j+X_j^+, L_j < t\leqslant L_j +
\ell_j+1 \bigr\}
\\
& &{}\cup\bigl\{ (x,t) \dvtx Z_j-X_{j-1}^- < x <
Z_j, u_{L_j}(x) < t\leqslant L_j +
\ell_j+1 \bigr\}.
\end{eqnarray*}
The first event considered is
\[
B_1:= \bigl[ |\nu_2| \leqslant1 \bigr].
\]

Notice that, conditioned on $u_{L_{j}}$, the number of points
$|\nu_2|$ is distributed as a Poisson random variable with mean
given by the area
$|R_2|$.
Now, on the event that $u_{L_j}$ satisfies (\ref{eq:staircase}),
\begin{eqnarray*}
|R_2| &\leqslant&(\ell_j+1) X_j^+ + (
\ell_{j-1}+1) X_{j-1}^- + (\ell_j+1)
X_{j-1}^- \leqslant3 (\ell_j+1) X_j^+ \\
&\leqslant&
C \frac{(\ell_j+1)^2}{N_j} \leqslant C \frac{1}{j^{3/4}}
\end{eqnarray*}
since
\[
\ell_j \leqslant C j^{1/4} \quad\mbox{and}\quad N_j
\geqslant L_j \geqslant Q_0+\cdots+Q_{j-1}
\geqslant C j^{5/4},
\]
and therefore
%
\begin{equation}
\label{eq:twopoints} \PP ( B_1|u_{L_j} ) \geqslant1 - C
|R_2|^2 \geqslant1 - C \frac{1}{j^{3/2}}.
\end{equation}
We also need the estimate to be positive for all $j$, which follows from
\[
\PP ( B_1|u_{L_j} ) \geqslant \PP ( \nu_2 =
\varnothing | u_{L_j} ) = e^{-|R_2|} \geqslant e^{-c} > 0.
\]

We now consider the events $B_2$ and $B_3$, which depend on $\nu_1$.
Define
%
\begin{equation}
\label{eq:aofx} A(x) = \int_{Z_j}^{x} \bigl[
L_j - u(z) \bigr] \,\dd z,\qquad x \geqslant Z_j,
\end{equation}
and write $\nu_1=\{(x_1,t_1),(x_2,t_2),(x_3,t_3),\ldots\}$ with
$x_0=Z_j<x_1<x_2<x_3<\cdots$.
By definition of $\nu_1$, we have that
$(A(x_n)-A(x_{n-1}))_{n=1,2,3,\ldots}$ are i.i.d. exponential random variables
with mean $1$, independent of $u_{L_j}$.
The events $B_2$ and $B_3$ are defined in terms of
$A(x_n),n=1,2,3,\ldots,$
whence the estimates on their probabilities are always uniform on $u_{L_j}$.

Consider the event
%
\begin{equation}
\label{eq:event2} B_2:= \bigl[ (1-\varepsilon) (\ell_j-1)
< A(x_{\ell_j-1}) < A(x_{\ell
_j}) < (1+\varepsilon)\ell_j
\bigr].
\end{equation}
By Chernoff's exponential bounds,
%
\begin{equation}
\label{eq:b2whp} \PP(B_2) \geqslant1-e^{-c_\varepsilon\ell_j}.
\end{equation}

Consider the event
%
\begin{equation}
\label{eq:event3} B_3:= \biggl[ A(x_n)-A(x_{n-1})
\leqslant\frac{\ell_j}{12} \mbox { for } n=1,2,\ldots,\ell_j
\biggr].
\end{equation}
By a simple union bound, we have
%
\begin{equation}
\label{eq:b3whp} \PP( B_3 ) \geqslant1 - \ell_j
e^{-\ell_j/12} \geqslant1 - C e^{-c
\ell_j}.
\end{equation}

Using (\ref{eq:b2whp}) and (\ref{eq:b3whp}), we get
\[
\PP( B_2 \cap B_3 ) \geqslant1 - C_\varepsilon
e^{-c_\varepsilon\ell_j}.
\]
Now, since $\ell_j \geqslant12$, we have
\[
\PP( B_2\cap B_3 ) \geqslant\PP \bigl( 1-\varepsilon<
A(x_n)-A(x_{n-1}) < 1 \mbox{ for } n=1,2,\ldots,
\ell_j \bigr) > e^{-c_\varepsilon\ell_j} > 0
\]
and thus adjusting $c_\varepsilon$ we get
\[
\PP( B_2 \cap B_3 ) \geqslant1 - e^{-c_\varepsilon\ell_j}.
\]

Since $\nu_1$ is conditionally independent of $\nu_2 \cup\nu_3$ given
$u_{L_j}$, we have that
\[
\PP ( B_1 \cap B_2 \cap B_3 |
u_{L_j} ) \geqslant p_j
\]
for
\[
p_j = \bigl( 1- e^{-c_\varepsilon\ell_j} \bigr) \bigl( e^{-c} \vee
\bigl(1-Cj^{-3/2}\bigr) \bigr).
\]

Notice that the sequence $(p_j)_{j=0,1,2,\ldots}$
satisfies (\ref{eq:prodpositive}), thus it only
remains to show that $B_1 \cap B_2 \cap B_3$ implies $A_j$.

Suppose $B_1$, $B_2$ and $B_3$ happen.
By (\ref{eq:staircase}) and monotonicity of $u$, we have
\[
{N_j} [x_n-x_{n-1}] \leqslant
{A(x_n)-A(x_{n-1})} \leqslant \bigl[ L_j -
u(x_{n}) \bigr] [ x_n - x_{n-1} ],
\]
whence by (\ref{eq:event2})
%
\begin{equation}
\label{eq:xelljnotfar} x_{\ell_j-1}-Z_j \leqslant x_{\ell_j}-Z_j
\leqslant (1+\varepsilon)\frac{\ell_j}{N_j}= X_j^+,
\end{equation}
and by (\ref{eq:event3})
%
\begin{equation}
\label{eq:pointsclose} x_n-x_{n-1} \leqslant \frac{\ell_j}{12 N_j}
\leqslant \frac{X_{j-1}^-}{3}.
\end{equation}
Moreover,
for $n=1,2,\ldots,\ell_j-1$,
\[
{A(x_n)-A(x_{n-1})} \leqslant \bigl[ L_j -
u(x_{\ell_{j}-1}) \bigr] [x_n-x_{n-1}]
\]
and, by (\ref{eq:event2}),
\[
x_{\ell_j} -Z_j \geqslant x_{\ell_j-1}
-Z_j \geqslant (1-\varepsilon)\frac{\ell_j-1}{L_j - u(x_{\ell_{j}-1})} \geqslant (1-
\varepsilon)\frac{\ell_j-1}{N_{j+1}} = X_j^-
\]
as long as $Z_{j+1} = S_{L_j+Q_j} \geqslant x_{\ell_j-1}$.

Therefore, to prove (\ref{eq:groupbehavior}) it suffices to show that
%
\begin{equation}
\label{eq:endswell} %
\cases{ x_{\ell_j-1} \leqslant S_{L_j+Q_j}
\leqslant x_{\ell_j},& \vspace*{2pt}
\cr
x_0 - X_{j-1}^-
\leqslant S_{L_j+n} < S_{L_j+Q_j},&\quad $n=1,2,\ldots,Q_j-1.$}
\end{equation}
The remainder of the proof is dedicated to proving (\ref{eq:endswell})
assuming (\ref{eq:xelljnotfar}), (\ref{eq:pointsclose}) and that
$B_1$ occurs.

We first recall that the points in $(x,t)\in\nu$ that correspond to customers
$( S_{L_j+1},S_{L_j+2},\ldots, S_{L_j+Q_j} )$ are such that
$u_{L_j}(x) < t \leqslant L_j+\ell_j+1$.
When these points are neither in $R_1$ nor in $R_2$, they must be in
$R_3$ given by $t \in ( u_{L_j}(x), L_j+\ell_j+1  ]$ and
%
\begin{equation}
\label{eq:pointsfar} x< Z_j-X_{j-1}^- \quad\mbox{or}\quad x >
Z_j+X_j^+.
\end{equation}
The points in $R_1$ are given by $(x_n,t_n)_{n=1,2,\ldots}$, and $R_2$
is either
empty or contains one point, denoted by $(x',t')$.

Let $n'$ be the maximal index between $0$ and $\ell_j$ such that
\[
( S_{L_j},S_{L_j+1},S_{L_j+2},\ldots, S_{L_j+n'-1},
S_{L_j+n'} ) = ( x_0, x_1, x_2,
\ldots, x_{n'-1}, x_{n'} ).
\]
If $n' = \ell_j$, we have $Q_j=\ell_j$, thus (\ref{eq:endswell})
is satisfied.
So suppose $n' \leqslant\ell_j-1$.
We claim that
%
\[
S_{L_j+n'+1}=x'
\]
with $x'$ satisfying
\[
x_{n'}-\frac{\ell_j}{12 N_j} \leqslant x' <
x_{n'+1},
\]
and moreover
%
\begin{equation}
\label{eq:nextpoints} S_{L_j+n+1}=x_n \qquad\mbox{for }
n=n'+1,n'+2,\ldots,\ell_j,
\end{equation}
that is, the points in $R_3$ cannot participate in the construction of
$S_{L_j+Q_j}$.

In the case $x'<x_{n'}$, we will have $Q_j = \ell_j+1$ and
$S_{L_j+Q_j}=x_{\ell_j}$.
Otherwise, $x_{n'}<x'<x_{n'+1}$, we will have $Q_j = \ell_j$, and in
this case
$S_{L_j+Q_j}=x_{\ell_j-1}$ if $n' \leqslant\ell_j-2$ or
$S_{L_j+Q_j}=x' \in(x_{\ell_j-1},x_{\ell_j})$ if $n'=\ell_j-1$.
Therefore, (\ref{eq:endswell}) is always satisfied.

It thus remains to prove the above claim.
By definition of $n'$, the point $({x}',{t}')\in\nu$ corresponding to
$S_{L_j+n'+1}$ cannot be in $R_1$.
But it cannot be in $R_3$ either.
Indeed, since
$S_{L_j+n'}=x_{n'}$ and
\[
x_{n'} < x_{n'+1} \leqslant x_{n'} +
\frac{\ell_j}{12 N_j},
\]
we must have
\[
x_0 - \frac{\ell_j}{12 N_j} \leqslant x_{n'} -
\frac{\ell_j}{12 N_j} \leqslant {x}' < x_{n'+1},
\]
thus ${x}'$ cannot satisfy (\ref{eq:pointsfar}).
Therefore, $(x',t')$ is the only point in $\nu_2$.

We finally show (\ref{eq:nextpoints}).
Start with $n=n'+1$.
Write $\tilde{x}=S_{L_j+n'+2}$, corresponding to a point
$(\tilde{x},\tilde{t})\in\nu$.
This point cannot be in $R_2$, since $(x',t')$ was the only such point.
As before,
\[
\bigl|x'-x_{n'+1}\bigr| \leqslant \bigl|x'-x_{n'}\bigr|
+ |x_{n'}-x_{n'+1}| \leqslant \frac{\ell_j}{6 N_j},
\]
thus we must have
\[
\tilde{x} < x_{n'+1} \leqslant x_{\ell_j} \leqslant
Z_j+X_j^+
\]
and
\[
\bigl|\tilde{x} - x'\bigr| < \frac{\ell_j}{6 N_j},
\]
whence
\[
\tilde{x} > x'-2\frac{\ell_j}{N_j} \geqslant x_0 -
\frac{\ell_j}{4 N_j}
\]
and again $\tilde{x}$ cannot satisfy (\ref{eq:pointsfar}) either.
Therefore, $(\tilde{x},\tilde{t})\in\nu_1$ which implies $\tilde
{x}=x_{n'+1}$.
For $n=n'+2,\ldots,\ell_j$, the argument is the same.

\section{Finite speed and random service times}
\label{sec:finitespeedrandomservice}

In this section, we show how the proof of Theorem~\ref
{thm:transientlog} for the
particular case $T=1,v=\infty$ can be adapted to more broad conditions as
stated in Section~\ref{sec:sec1introduction}.
We start describing the analogous construction for the stochastic
evolution of
potentials.
Assume that at time $t=0$ the server starts serving a customer at $x=0$
(for convenience, we consider here the potentials corresponding to
times when
service starts).
Assume also that the set of waiting customers is given by a Poisson Point
Process on $\RR$ with intensity $-u_0(x) \,\dd x$ for a unimodal
potential $u_0$
with maximal value $u_0(0)=0$.
In analogy with (\ref{eq:randomupdate}), given $w=(T,E,U)$ we define the
operator $\H_w$ by
\begin{eqnarray*}
\int_{x_*-z}^{x_*+z}(\M+T-u)\,\dd x &=& E,\qquad
\begin{array} {l} a=\M+T-u(x_*-z),
\\
b=\M+T-u(x_*+z), \end{array} \\
 x^{*} &=& %
\cases{
x_* - z, &\quad $\mbox{if }\displaystyle U \in\biggl(0, \frac{a}{a+b}\biggr]$,\vspace*{2pt}
\cr
x_* + z, &\quad
$\mbox{if }\displaystyle U \in\biggl(\frac{a}{a+b},1\biggr)$, } %
\end{eqnarray*}
and
\[
\bigl(\H_w (u) \bigr) (x) = %
\cases{\displaystyle \M+T+
\frac{z}{v}, &\quad $x=x^*$,\vspace*{2pt}
\cr
\M+T, & \quad $x\in[x_*-z,x_*+z], x\ne
x^*$,\vspace*{2pt}
\cr
u(x), & \quad $\mbox{otherwise.}$} %
\]
Notice that $\M(\H_w (u))=\M(u)+T+\frac{z}{v}$, $S(\H_w
(u))=x^{*}$, and
$\int_\RR[\H_w (u)-u]\,\dd x = E$.

We take an i.i.d. sequence $(\omega_n)_{n=1,2,\ldots}$, where each
$\omega_n=(T_n,E_n,U_n)$ has independent coordinates, distributed respectively
as the service time, a standard exponential, and a uniform on $[0,1]$.
We define $u^k_n$ by (\ref{eq:ukn}) and let $u_n=u^0_n$ and $u=u_0$.

Define $t_n = \M(u^0_n)$ and $S_n=S(u^0_n)$.
In analogy with Lemma~\ref{lemma:coupling}, we have

\begin{lemma}
The pair sequence $(t_n,S_n)_{n=1,2,\ldots}$ described above has the
same distribution
as $(t_n,{\mathcal S}_{t_n})_{n=1,2,\ldots}$ given by the beginning
of service times and the corresponding positions.
\end{lemma}

For the evolution $(u_n)_{n=0,1,2,\ldots}$ we will define a sequence of stopping
times $0=\N_0<\N_1<\N_2<\cdots$ in $\NN_0$, as well as the
corresponding events of success
$A_j$ defined in terms of $u_{\N_j}$ and whose occurrence
is determined by $u_0,u_1,\ldots,u_{\N_{j+1}}$.

The construction will have the following properties.
For some sequence $p_j$ and any $u$ that is centered and unimodal,
%
\begin{equation}
\label{eq:blockestimateprobab} \PP^{u} ( A_{j} | u_{\N_j} )
\geqslant p_j \qquad\mbox{on } A_0 \cap\cdots\cap
A_{j-1}\quad \mbox{and}\quad \prod_j
p_j > 0.
\end{equation}
Moreover, the event $\bigcap_{j=0}^\infty A_j$ implies $S_{n} \sim
_\varepsilon\sigma\log n$
just as in the proof of Theorem~\ref{thm:transientlog} in the end of
Section~\ref{sec:potential}.

Step $0$ provides $\sigma=\pm1$ which indicates the direction in which
subsequent blocks are supposed to grow.
In the steps described below we let
$\N_{j}=Q_0+\cdots+Q_{j-1}$, where $Q_j$ is the number of customers
served in each block,
$L_j = \M(u_{\N_j})$ the physical time,
$N_j = L_j-u(S_{\N_j})$ the height of discontinuity in the potential,
$Z_j=\sigma S_{\N_j}$,
$X_{j} = Z_{j+1} - Z_{j}$ physical displacement during each block,
and $M_j=L_{j+1}-L_j$ is the time elapsed within each block.
In this setting, the time $L_{j+1}$ is given by the instant when the server
reaches the customer located at $\sigma Z_{j+1}$, and the next block starts.

Let $\ell_j$ be given as above, and write $m_j = \ell_1+\cdots+\ell_j$.
Let $j_*$ be such that
\[
\frac{1}{m_{j_*}}\cdot\frac{1}{v} < \frac{1}{16} \quad\mbox{and}\quad \PP
\biggl(\frac{n}{2} < T_1+\cdots+T_{n} < 2n \biggr)>0\qquad
\mbox{for all } n>\ell_{j_*}.
\]
Fix $m=1+m_{j_*}$.
In steps $1,\ldots,j_*$ we will relax the lower bound on time and take
$M_j^-=0$.
This is compensated by finding a big number of customers at step $0$,
namely $Q_0=m$.
So the triggering step will take care of however small the speed $v$
is, as well as complications arising from the distribution of $T$.

The event $A_0$ is defined by the following conditions.
First, that $S_m$ is an unexplored point, that is, $\sigma S_m > \sigma
S_n$ for all $n=0,1,2,\ldots,m-1$ and some $\sigma=\pm1$.
Second, $M_0 \geqslant M_0^- = m_{j*}$.
Finally,
\[
T_{m} \leqslant1 \quad\mbox{and}\quad X_0^- \leqslant|S_m
- S_{m-1}| \leqslant X_0^+,
\]
where $X_0^- = \frac{4}{N_1}$ and $X_0^+=v$.

We claim that $\PP^u(A_0)\geqslant p_0$ for some $p_0>0$ that does not
depend on $u$.
To prove the claim, consider the following events.
First, suppose $T_1 \geqslant1$.
Suppose also that $E_1 \geqslant16$, and $U_1$ is such that $S_1$ lies
on the bigger side of $-u$, as in Section~\ref{sec:blocks}.
Assuming that all these happen, $\sigma$ is determined by which
direction $S_1$ was found, that is, the higher side of the \textsl
{plateau} in $u_1$.
In the sequel we assume for simplicity that $\sigma=+1$, otherwise
mirror the system around $x=0$.
Now suppose that, for all $n=2,\ldots,m-1$, $T_n \geqslant1$, $E_n \in
[0,1]$, and $U_n > \frac{1}{2}$.
Finally, suppose that $T_m \leqslant1$, $8 \leqslant E_m < 16$ and
$U_m > \frac{1}{2}$.

Let us show that these events imply $A_0$, which proves our claim.
By assumption $S_1>0$ and
\[
16 \leqslant E_1 = \int_{-S_1}^{S_1}
\bigl[T_1 - u(x)\bigr] \,\dd x \leqslant-2 S_1 \cdot
u(S_1)
\]
and thus
\[
-u(S_1) \geqslant\frac{8}{S_1}.
\]
Writing $z_2=|S_2-S_1|$, we have
\[
1 \geqslant E_2 = \int_{S_1-z_2}^{S_1+z_2}
\bigl[\M(u_1)+T_2-u_1(x)\bigr] \,\dd x \geqslant
\int_{S_1}^{S_1+z_2} \bigl[-u_1(x)\bigr] \,\dd
x \geqslant-u(S_1) \cdot z_2
\]
and thus $z_2 \leqslant\frac{S_1}{8}$ and $0<S_1-z_2<S_1$.
Since $u_1(x) \leqslant0$ for $x>S_1$ and $u_1(x)=T_1>0$ for
$-S_1<x<S_1$, the choice of $U_2 > \frac{1}{2}$ implies that $S_2=S_1+z_2>S_1$.
By the same argument, $E_3 \leqslant1$ implies $|S_3-S_2| \leqslant
\frac{S_1}{8}$, and thus $U_3 > \frac{1}{2}$ implies $S_3>S_2$, and
so on.
Therefore, $S_{m-1}>S_1$ and $u_{m-1}(x) = u(x)$ for $x>S_{m-1}$.
As before, writing $z_m=|S_m-S_{m-1}|$ we have
\begin{eqnarray*}
16 &>& E_m = \int_{S_{m-1}-z_m}^{S_{m-1}+z_m} \bigl[\M
(u_{m-1})+T_m-u_{m-1}(x)\bigr] \,\dd x \\
&\geqslant&
z_m \cdot\bigl[T_1+\cdots+T_m -
u(S_{m-1})\bigr],
\end{eqnarray*}
thus $z_m < 2S_1$, and since $U_m>\frac{1}{2}$ we have $S_m>S_{m-1}$.
Moreover, $T_1+\cdots+T_m \geqslant m_{j_*}$ and thus $S_m - S_{m-1} =
z_m < v = X_0^+$.
Finally,
\begin{eqnarray*}
8 &\leqslant& E_{m} = \int_{S_{m-1}-z_m}^{S_{m-1}+z_m}
\bigl[\M (u_{m-1})+T_m-u_{m-1}(x)\bigr] \,\dd x
\leqslant2z_m \cdot\bigl[\M(u_m) - u(S_m)
\bigr]\\
& =& 2z\cdot N_1,
\end{eqnarray*}
and thus $S_m-S_{m-1} = z_m \geqslant X_0^-$. This proves the above claim.

For $j\geqslant1$, we define $Q_j$ and the event $A_j$ as in
Section~\ref{sec:blocks}, with condition (\ref{eq:groupbehavior})
replaced by
\[
\cases{ Q_j = \ell_j \mbox{ or }
\ell_j+1, \vspace*{2pt}
\cr
X_j^- \leqslant
X_j \leqslant X_j^+, \vspace*{2pt}
\cr
M_j^-
\leqslant M_j \leqslant M_j^+, \vspace*{2pt}
\cr
Z_{j-1} < S_n < Z_{j+1},&\quad $\mbox{for }
\N_j \leqslant n < \N_{j+1},$} %
\]
where
$
X_j^- = (1-\varepsilon) \frac{\ell_j-1}{N_{j+1}},
X_j^+ = (1+\varepsilon) \frac{\ell_j}{N_j},
M_j^+ = 3 \ell_j + 3
$
and $M_j^- = \frac{1}{2} \ell_j \I_{j>j_*}$.

Assuming that $(A_0 \cap\cdots\cap A_{j-1})$ occurs,
since $u_{\N_j}$ is unimodal it must satisfy
\[
\cases{ \M(u_{\N_{j}}) = L_{j},
\vspace*{2pt}
\cr
u_{\N_{j}}(x) = u_0(x) \leqslant
L_{j} - N_{j}, &\quad $\mbox{for } x > Z_{j}$,
\vspace*{2pt}
\cr
u_{\N_{j}}(x) \geqslant L_{j} -
M_{j-1}, &\quad $\mbox{for } Z_{j} - X_{j-1}^- < x <
Z_{j}$.} %
\]
For $j=1$ the last condition is replaced by
$u_{\N_{1}}(x) \geqslant L_{1} - T_m - \frac{X_0^+}{v} \geqslant L_1
- 2$.
Moreover,
\[
N_j = -u(\sigma Z_j)+L_j \geqslant
L_j \geqslant M_0^-+\cdots+M_{j-1}^- \geqslant
m_j \geqslant c j^{5/4}
\]
and thus
$X_j^+ \leqslant C j^{-1}$ and $X_{j-1}^- \leqslant C j^{-1}$.
Therefore, $M_j^+ \leqslant C j^{1/4}$.

To estimate $\PP^{u}(A_j|u_{\N_j})$ on $(A_0 \cap\cdots\cap
A_{j-1})$, we consider the events $B_1$, $B_2$, and $B_3$ as in
Section~\ref{sec:finitespeedrandomservice}.

The region analogous to $R_2 \subseteq\RR\times\RR$ is contained in
the union
of the rectangles
$[Z_j-X_{j-1}^-,Z_j] \times
[L_j-M_{j-1}^+,L_j+M_j^+]$ and $[Z_j,Z_j+X_j^+]\times[L_j,L_j+M_j^+]$
with $M_0$ replaced by $T_m+\frac{X_0^+}{v} \leqslant2$ for $j=1$.
The above inequalities imply that
$|R_2| \leqslant(X_{j-1}^- + X_j^+)(M_{j-1}^+ + M_{j}^+) \leqslant C j^{-3/4}$.
Therefore,
\[
\PP^u( B_1 | u_{\N_j}) \geqslant \biggl( 1 -
\frac{C}{j^{3/2}} \vee e^{-c} \biggr).
\]

Events $B_2$ and $B_3$ are defined by (\ref{eq:event2}) and (\ref{eq:event3}).
Therefore, occurrence of $B_2 \cap B_3$ implies inequalities (\ref
{eq:xelljnotfar}) and (\ref{eq:pointsclose}), and its probability
satisfies (\ref{eq:blockestimateprobab}).
The desired bounds for $S_n$ for $\N_j \leqslant n < \N_{j+1}$ and for
$X_j$ thus follow exactly as in Section~\ref{sec:blocks}.

It remains to control $M_j$, which was not necessary in the case
$T=1,v=\infty$ because $M_j=Q_j$ in that setup.
But $M_j$ is composed of $Q_j$ service times plus traveling time.
The latter is nonnegative and bounded by
\[
2\frac{X_j}{v} \leqslant 2\frac{X_j^+}{v} \leqslant 2 \frac{\ell_j/N_j}{v}
\leqslant 2 \frac{\ell_j/m}{v} \leqslant \ell_j.
\]
Therefore, the inequality $M_j^- \leqslant M_j \leqslant M_j^+$ holds whenever
the sum of $Q_j$ service times is bigger than $\frac{1}{2}Q_j \I
_{j>j_*}$ and less than $2Q_j$.
The probability of this event is exponentially high in $\ell_j$, and
positive by the choice of $j_*$.
This completes the proof of Theorem~\ref{thm:transientlog}.

\section*{Acknowledgment}
This work has been started at IMPA in February 2011, and
S.~Foss thanks the institute for the hospitality.

%



\printaddresses
\end{document}